\documentclass[12pt]{article}
\usepackage{mathrsfs}
\usepackage{amssymb,amsmath,amsthm,amsfonts}
\usepackage{color}
\def\sqr#1#2{{\vcenter{\vbox{\hrule height.#2pt
              \hbox{\vrule width.#2pt height#1pt \kern#1pt \vrule width.#2pt}
          \hrule height.#2pt}}}}
\def\a{\alpha}
\def\b{\Pi}

\def\sqr#1#2{{\vcenter{\vbox{\hrule height.#2pt
              \hbox{\vrule width.#2pt height#1pt \kern#1pt \vrule width.#2pt}
              \hrule height.#2pt}}}}
\def\3n{\negthinspace \negthinspace \negthinspace }
\def\2n{\negthinspace \negthinspace }
\def\1n{\negthinspace }

\def\={\buildrel \triangle \over =}

\def\a{\alpha}
\def\b{\beta}

\def\min{\mathop{\rm min}}

\def\inf{\hbox{\rm inf$\,$}}

\def\|{\Big |}
\def\({\Big (}
\def\){\Big )}
\def\[{\Big[}
\def\]{\Big]}
\def\be{\begin{equation}}
\def\bel{\begin{equation}\label}
\def\ee{\end{equation}}
\def\bt{\begin{theorem}}
\def\bcd{\begin{condition}}
\def\ecd{\end{condition}}
\def\et{\end{theorem}}
\def\bc{\begin{corollary}}
\def\ec{\end{corollary}}
\def\bde{\begin{definition}}
\def\ede{\end{definition}}
\def\bl{\begin{lemma}}
\def\el{\end{lemma}}
\def\bp{\begin{proposition}}
\def\ep{\end{proposition}}
\def\br{\begin{remark}}
\def\er{\end{remark}}
\def\ba{\begin{array}}
\def\ea{\end{array}}
\def\ed{\end{document}}

\def\square#1{\vbox{\hrule\hbox{\vrule height#1%
     \kern#1\vrule}\hrule}}
\def\rectangle#1#2{\vbox{\hrule\hbox{\vrule height#1%
     \kern#2\vrule}\hrule}}

\font\tenbb=msbm10 \font\sevenbb=msbm7 \font\fivebb=msbm5

\newfam\bbfam
\scriptscriptfont\bbfam=\fivebb \textfont\bbfam=\tenbb
\scriptfont\bbfam=\sevenbb

\newtheorem{lemma}{Lemma}
\newtheorem{remark}{Remark}

\newtheorem{theorem}{Theorem}
\newtheorem{corollary}{Corollary}

\newtheorem{definition}{Definition}
\newtheorem{proposition}{Proposition}
\newtheorem{condition}{Assumption}

\begin{document}

\title{ Risk-Sensitive Mean-Field-Type Control\thanks{ Alain Bensoussan is also with the College of Science and Engineering, Systems Engineering and Engineering Management, City University Hong Kong.}}
\author{Alain Bensoussan\\ 
International Center for Decision and Risk Analysis\\
Jindal School of Management, University of Texas at Dallas\\
Boualem Djehiche\\
Department of Mathematics, KTH Royal Institute of Technology, Stockholm\\
Hamidou Tembine\\
Learning \& Game Theory Laboratory, New York University Abu Dhabi\\
Phillip Yam\\
Department of Statistics, The Chinese University of Hong Kong}

\maketitle

\begin{abstract}
We study risk-sensitive optimal control of a  stochastic differential equation (SDE) of mean-field type, where the coefficients are allowed to depend on some functional of the law as well as the state and control processes. Moreover the risk-sensitive cost functional is also of mean-field type. We derive optimality equations in infinite dimensions connecting dual functions  associated with Bellman functional to the adjoint process of the Pontryagin maximum principle.  The case of linear-exponentiated quadratic cost and its connection with the risk-neutral solution is discussed.
\end{abstract}

\section{Introduction}
We consider the mean-field-type control problem with a risk-sensitive performance functional. For mean-field-type control, the approach is generally to use the maximum principle, see for instance \cite{one,three,five,seven,eight}. We refer the reader to \cite{survey} for a recent survey on the approach. In \cite{two}, it is shown that one can introduce a system of dual Hamilton-Jacobi-Bellman and Fokker-Planck equations, dHJB-FP, similar to that introduced by Lasry \& Lions \cite{six} to handle risk-sensitive mean-field-type control problem. However, the solution of the dHJB equation is not the value function, but must be interpreted as an adjoint function for a dual control problem. Here, we extend this approach to the risk-sensitive mean-field-type control problem. We then make the connection with the stochastic maximum principle, and study the linear-exponentiated-quadratic case. 

We consider functions $f(x,m,v), g(x,m,v)$ where the arguments are $x \in\mathbb{R}^n$, $m$ is a probability measure on $\mathbb{R}^n$, but we will remain mostly in the regular case (with respect to Lebesgue measure), in which $m$ represents the probability density, assumed to be in $L^2(\mathbb{R}^n)$ and $v$ is a control in $\mathbb{R}^d.$ 
The function $f$ is scalar, and the function $g$ is a vector in $\mathbb{R}^n.$ We also consider  $\sigma: x \in \mathbb{R}^n \mapsto \sigma(x)\in  \mathcal{L}(\mathbb{R}^n;\mathbb{R}^n),$ and $h  : \ (x,m)\in \mathbb{R}^n\times L^2(\mathbb{R}^n) \mapsto h(x,m)\in \mathbb{R}.$ All these functions are smooth. In the case of the differentiability with respect to the measure $m,$ we use the concept of Gateaux differentiability. If  $F : L^2(\mathbb{R}^n) \rightarrow \mathbb{R},$ 
then
$$
\lim_{\theta \rightarrow 0}\ \frac{d}{d\theta} \ F(m+\theta \tilde{m})= \int_{\xi\in \mathbb{R}^n} \frac{\partial F}{\partial m} (m)(\xi) \tilde{m}(\xi)d\xi,
$$
with the functional $\xi\in \mathbb{R}^n   \mapsto \frac{\partial F}{\partial m}(m)(\xi)\in L^2(\mathbb{R}^n).$

Consider a  probability space $(\Omega, \mathcal{X},P)$  and a filtration $\mathbb{F}=(\mathcal{F}_t)_{t\ge 0}$, generated by a Wiener process $w(\cdot)$ in $\mathbb{R}^n$.
The classical mean-field-type control problem is the following: Given a control process $v(\cdot)$ adapted to the filtration $\mathbb{F}$, the corresponding state equation is the McKean-Vlasov equation of the mean-field type:
\begin{equation}\label{one0}
dx(t)=g(x,m_v,v)dt+\sigma(x(t))dw(t),\ x(0)=x_0,
\end{equation}
in which $m_v(x,t)$ is the probability density of the random variable (state) $x(t).$ The initial value $x_0$ is a random variable that is 
independent of the Wiener process $w(\cdot)$.  This density is well-defined if the matrix $a(x)=\sigma(x)\sigma^*(x)$ is invertible. 
We define the second order differential operator
$$A\phi(x)=-\frac{1}{2}\sum_{i,j}a_{ij}(x)\frac{\partial^2\phi(x)}{\partial x_i\partial x_j},$$ and its adjoint
$$A^*\phi(x)=-\frac{1}{2}\sum_{i,j}\frac{\partial^2[a_{ij}(x)\phi(x)]}{\partial x_i\partial x_j}.$$

Next define the cumulative  expected cost functional $J(v(.))$ as
\begin{equation} \label{one}
=\mathbb{E}\left[  \int_0^T f(x,m_v, v)dt + h(x(T),m_v(T))\right].
\end{equation}
The risk-neutral mean-field-type control problem is to minimize  $J(v(\cdot)).$

In this paper we consider a risk-sensitive cost functional, which means that we replace (\ref{one}) by
\begin{equation} \label{two}
J^{\alpha}(v(.))=\mathbb{E}e^{\alpha\left[  \int_0^T f(x(t),m_v(t), v(t))dt + h(x(T),m_v(T))\right]}
\end{equation}
in which $\alpha$ is a real number, representing the risk-sensitivity index of the decision-maker. When $\alpha > 0,$
 it models a risk-averse decision, when $\alpha<0$ a risk-seeker individual.

Note that 
$$
\frac{J^{\alpha}(v(.))-1}{\alpha} \rightarrow J(v(.)),\ \mbox{ as }\ \alpha \rightarrow 0.
$$
Equivalently, 
$$
  \frac{1}{\alpha}\log_{10}(J^{\alpha}) \rightarrow J(v(.)),\ \mbox{ as }\ \alpha \rightarrow 0.
$$
So the case (\ref{one}) is considered as representing the risk-neutral situation corresponding to $\alpha=0$.

From now on, we shall assume that $\alpha>0.$ The case $\alpha<0$ is examined using the  same methodology by change $f \rightarrow -f, $ and $h \rightarrow -h.$

\section{Risk-Neutral Case}
We define the risk-neutral Hamiltonian $H:\,\,(x,m,q)\in \mathbb{R}^n \times L^2(\mathbb{R}^n)\times \mathbb{R}^n 
\mapsto H(x,m,q)\in \mathbb{R}  $  as
$$
H(x,m,q)=\inf_v\{ f(x,m,v)+q g(x,m,v)\}
$$
and the optimal value of $v$ is denoted by $v^*(x,m,q).$ We then set $G(x,m,q)=g(x,m,v^*(x,m,q)).$

The mean-field-type control problem is easily transformed into a stochastic control problem for a higher dimensional state, which is the 
probability density $m_v(\cdot)$. It is the solution of the Fokker-Planck equation 
\begin{equation}\label{fpk}
\frac{\partial m_v}{\partial t}+A^*m_v+div(g(x,m_v,v(\cdot)) m_v)=0,\ 
\end{equation} and  $m_v(0,x)=m_0(x),$
in which $m_0(x)$  is the probability density of the initial value $x_0.$ The objective functional $J(v(.))$ can be written as
\begin{eqnarray} \label{onebis}
J(v(.))=  \int_0^T \int_{x\in \mathbb{R}^n}f(x,m_v(t), v(t)) m_v(x,t)dxdt  \nonumber \\  + \int h(x,m_v(T)) m_v(x,T)dx.
\end{eqnarray}

The adjoint system of optimality associated with   (\ref{fpk}) and  (\ref{onebis}) is given by

\begin{eqnarray}\label{bell1}
u(x,T)=h(x,m)+\int_{\mathbb{R}^n} \frac{\partial }{\partial m}h(\xi,m)(x)m(\xi,T)d\xi \\ 
-\frac{\partial u}{\partial t}+Au=H(x,m,D_xu(x))\nonumber \\ \label{bell2} +\int_{\mathbb{R}^n} \frac{\partial }{\partial m}H(\xi,m,D_xu(\xi))(x)m(\xi,t)d\xi  \\ \label{bell21}
\frac{\partial m}{\partial t}+A^*m+div(G(x,m,D_xu) m)=0,\\ m(0,x)=m_0(x).
\end{eqnarray}
The optimal feedback control is $v^*(x,t)=v^*(x,m,D_xu)$. When the functions $h,f,g$ are mean-field free, i.e., do not depend on $m$ then the equations (\ref{bell1}) and  (\ref{bell2}) reduces to standard  Hamilton-Jacobi-Bellman equation in $x$  and a  Fokker-Planck-Kolmogorov equation:

\begin{eqnarray}\label{bell3}
u(x,T)=h(x), \\ \label{bell4}
-\frac{\partial u}{\partial t}+Au=H(x,D_xu(x)),  \\
\frac{\partial m}{\partial t}+A^*m+div(G(x,D_xu) m)=0,\\ m(0,x)=m_0(x).
\end{eqnarray}
In this case, $u$ can be interpreted as the value function, the optimal feedback $v^*(x, t) = v^*(x, D_xu(x, t))$ is time consistent, which means that it does not depend on the initial condition of the dynamic system (\ref{one0}), whereas when the system is coupled, it does.

\section{Risk-Sensitive Case}
\subsection{Mean-Field Free Case}
Let us consider the problem of optimizing
\begin{equation} \label{rsmf0}
J^{\alpha}(v(\cdot))=\mathbb{E}e^{\alpha\left[  \int_0^T f(x(t), v(t))dt + h(x(T))\right]},
\end{equation}
subject to the state dynamics \begin{equation}\label{statemf0}
dx(t)=g(x(t),v(t))+\sigma(x(t))dw(t), \quad  x(0)=x_0,
\end{equation}
To be able to apply the optimality principle we have to introduce a second state equation, namely
 \begin{equation}\label{statemf0}
dz=f(x(t),v(t))dt,\quad z(0)=0,
\end{equation}
and we then get 
$$ 
J^{\alpha}(v(\cdot))=\mathbb{E}e^{\alpha\left[z(T)+ h(x(T))\right]}.
$$
In this way the functional involves only the final state, but the state is now augmented. The new state is  the pair $(x(t),z(t))$.  However, we are in the standard situation, in which we can apply Dynamic Programming. Introduce the family of problems
\begin{eqnarray}\label{bell5}
dx=gds+\sigma(x)dw,  \ \ x(t)=x\\ \label{bell6}
dz=f(x(s),v(s))ds,  \ \ z(t)=z.
\end{eqnarray}
We denote the solution by $x_{x,z,t}(s),z_{x,z,t}(s),\ $ and set $J^{\alpha}_{x,z,t}(v)=\mathbb{E} e^{\alpha[z_{x,z,t}(T)+h(x_{x,z,t}(T))]}.$ We define
$\Phi^{\alpha}({x,z,t})=\inf_{v(.)}\ J^{\alpha}_{x,z,t}(v).$ We can then write the Bellman equation:

\begin{eqnarray}\label{bell7}
-\frac{\partial \Phi^{\alpha}}{\partial t}+A\Phi^{\alpha}=\inf_{v} D_x\Phi^{\alpha}g(x,v)+ \frac{\partial \Phi^{\alpha}}{\partial z}(x,v), \\ \label{bell8}
\Phi^{\alpha}(x,z,T)=  e^{\alpha[z+h(x)]}.
\end{eqnarray}
The above system can be  solved by separation of variables as follows: 
$\Phi^{\alpha}(x,z,t)=e^{\alpha z}u^{\alpha}(x,t)$ with $u^{\alpha}(x,t)$ solution of 

\begin{eqnarray}\label{bell10}
-\frac{\partial u^{\alpha}}{\partial t}+Au^{\alpha}=\inf_{v} D_x\Phi^{\alpha}. g(x,v)+ \alpha u^{\alpha} f(x,v), \\ \label{bell11}
u^{\alpha}(x,T)=  e^{\alpha h(x)}.
\end{eqnarray}

We see easily that $\frac{u^{\alpha}-1}{\alpha} \rightarrow u$ as $\alpha \rightarrow 0,$ where the function $u$ is the solution of the risk-neutral system  (\ref{bell3})-(\ref{bell4}). The optimal control is obtained by a feedback depending on the state $x$, but not on the state $z.$

\subsection{Mean-Field Dependence}
We now turn to the risk-sensitive mean-field-type control problem with (\ref{one0}) and  (\ref{two}). We introduce again a new state $z(t)$ with the mean-field term:
\begin{eqnarray}\label{bell12}
dx=g(x,m_v,v) dt+\sigma(x)dw, x(0)=x_0,\\ \label{bell13}
dz=f(x,m_v,v)dt, \ \ z(0)=0.\end{eqnarray}
We have to consider a feedback $v(x, z)$ depending on the full state $(x, z)$. The simplification which occurred
in the case without mean-field, namely the optimal feedback was depending on the state $x$ only, does not extend in the current context. However, we still consider that the probability $m_v(t)$ entering in the functions $f$ and $g$ is the probability density of $x(t)$ and not the joint probability distribution $\mu_v(x,z,t)$ of the pair $(x(t),z(t))$. Therefore,
$$
m_v(x,t)=\int_{\mathbb{R}}\mu_v(x,z,t)dz.
$$
The joint probability distribution $\mu_v(x,z,t)$ of the pair $(x(t),z(t))$ solves the degenerate Fokker-Planck-Kolmogorov equation

\begin{equation}\begin{array}{lll}\label{fpk3}
\frac{\partial \mu_v}{\partial t}+A^*\mu_v+div(g(x,m_v,v(x,z,t)) \mu_v)\\ +\frac{\partial}{\partial z}[ f(x,m_v,v(x,z)) \mu_v]=0,\\ \\ \mu_v(0,x,z)=m_0(x)\otimes\delta_{0}(z)
\end{array}
\end{equation}
and we can write the cost functional as

$$J^{\alpha}(v)=\int_{x\in \mathbb{R}^n}\int_{z\in \mathbb{R}}\ \mu(x,z,T)\ e^{\alpha[z+h(x,\int_{\mathbb{R}}\mu_v(x,z,T)dz]}dz dx
$$

We can apply the general theory by adapting the system  (\ref{bell1})  (\ref{bell2}) and (\ref{bell21}). We introduce the Hamiltonian of the augmented state as
$$
\tilde{H}(x,m,q,\rho)=\inf_{v}\{ q g+\rho f  \},
$$
reserving the terminology $H(x,m,q)=\tilde{H}(x,m,q,1).$ The optimal feedback control is
 denoted by $v^*(x, m, q, \rho)$ and we set
 $
 \tilde{F}(x,m,q,\rho)=f(x,m,v^*(x, m, q, \rho)),$ $   \tilde{G}(x,m,q,\rho)=g(x,m,v^*(x, m, q, \rho)).
 $

We also set
$
v^*(x, m, q)=v^*(x, m, q, 1),$  $$F(x,m,q)= \tilde{F}(x,m,q,1),\ G(x,m,q)= \tilde{G}(x,m,q,1).
$$

The risk-sensitive adjoint system is

\begin{eqnarray}
u(x,z,T)=e^{\alpha[z+h(x,m(T))]}\nonumber \\ \label{adjoint1} +\alpha \int_{\mathbb{R}^n}\int_{\mathbb{R}}   e^{\alpha[\zeta+h(\xi,m(T))]}
\frac{\partial }{\partial m}h(\xi,m)(x)\mu(\xi,\zeta,T)d\xi d\zeta \\  \nonumber
-\frac{\partial u}{\partial t}+Au=\tilde{H}(x,m,D_xu,\frac{\partial u}{\partial z})+\\ \label{adjoint2}
\int_{\mathbb{R}^n} \frac{\partial }{\partial m}\tilde{H}(\xi,m,D_xu(\xi,\zeta),\frac{\partial u}{\partial z}(\xi,\zeta))(x)\mu(\xi,\zeta,t)d\xi\ d\zeta  \\ 
\frac{\partial \mu}{\partial t}+A^*\mu+div(g(x,m,v(x,z,t)) \mu)+\nonumber \\ \label{adjoint3} \frac{\partial}{\partial z}[ f(x,m_v,v(x,z)) \mu]=0,\ \\  \label{adjoint4}
\mu(0,x,z)=m_0(x)\otimes\delta_{0}(z),\\  \label{adjoint5}
m(x,t)=\int_{\mathbb{R}}\mu_v(x,z,t)dz.
\end{eqnarray}

\subsection{Transformation of the equation}
We aim to transform the system  (\ref{adjoint1}),  (\ref{adjoint2}),  (\ref{adjoint3}). We introduce $\chi(x,z,t)$ defined by
$\chi(x,z,t)=\frac{\partial u}{\partial z}(x,z,t).$

We differentiate the equation in $u$, in (\ref{adjoint2}), with respect to $z$. Taking account of the fact that the integrals depend only on $x,$ we get the relation

\begin{eqnarray}\label{chione}
\chi(x,z,T)=\alpha e^{\alpha[z+h(x,m(T))]},\\ 
-\frac{\partial \chi}{\partial t}+A\chi \nonumber \\  \label{chiotwo} = \frac{\partial \chi}{\partial z} \tilde{F}(x,m,D_xu,\frac{\partial u}{\partial z})+D_x\chi. \tilde{G}(x,m,D_xu,\frac{\partial u}{\partial z}).
\end{eqnarray}
We have the following result
\begin{lemma}\label{chi-positive}
The function $(x,z,t) \mapsto \chi(x,z,t)$ is positive: $\chi>0.$
\end{lemma}
\begin{proof}
We proceed only formally, since the assumptions have not been stated. The equation (\ref{chiotwo}) is a linear parabolic equation in $\chi$, with no right-hand side and strictly positive final condition. This implies the result.
\end{proof}
This allows to assert that
$$
 \tilde{H}(x,m,D_xu,\frac{\partial u}{\partial z})=\chi  {H}(x,m,\frac{D_xu}{\chi} ),
 $$

$$\tilde{F}(x,m,D_xu,\frac{\partial u}{\partial z})=F(x,m,\frac{D_xu}{\chi}),$$

$$
\tilde{G}(x,m,D_xu,\frac{\partial u}{\partial z})=G(x,m,\frac{D_xu}{\chi}).$$

So we can write the system

\begin{eqnarray}
u(x,z,T)=e^{\alpha[z+h(x,m(T))]}+\nonumber \\ \label{adjoint11} \alpha \int_{\mathbb{R}^{n+1}}   e^{\alpha[\zeta+h(\xi,m(T))]}
\frac{\partial }{\partial m}h(\xi,m)(x)\mu(\xi,\zeta,T)d\xi d\zeta \\ 
-\frac{\partial u}{\partial t}+Au=\chi {H}(x,m,\frac{D_xu}{\chi})+\nonumber \\ \label{adjoint21}
\int_{\mathbb{R}^n} \frac{\partial }{\partial m}\chi(\xi,\zeta) {H}(\xi,m,\frac{D_xu(\xi,\zeta)}{\chi(\xi,\zeta)})(x)\mu(\xi,\zeta,t)d\xi\ d\zeta  \\ 
\frac{\partial \mu}{\partial t}+A^*\mu+div(G(x,m,\frac{D_xu}{\chi}) \mu)+\nonumber \\ \label{adjoint31} \frac{\partial}{\partial z}[ F(x,m,\frac{D_x u}{\chi}) \mu]=0,\ \\  \label{adjoint41}
\mu(0,x,z)=m_0(x)\otimes\delta_{0}(z),\\  \label{adjoint51}
m(x,t)=\int_{\mathbb{R}}\mu_v(x,z,t)dz,\\
\chi(x,z,T)=\alpha e^{\alpha[z+h(x,m(T))]},\\ 
-\frac{\partial \chi}{\partial t}+A\chi= \nonumber\\  \label{adjoint61} \frac{\partial \chi}{\partial z} {F}(x,m,\frac{D_xu}{\chi})+D_x\chi. {G}(x,m,\frac{D_xu}{\chi} ).
\end{eqnarray}
\begin{remark}
Note formally that $u^{\alpha},\chi^{\alpha}, \mu^{\alpha}, m^{\alpha}$ 
 the solution of the system (\ref{adjoint11}),(\ref{adjoint21}), (\ref{adjoint31}), (\ref{adjoint41}), (\ref{adjoint51}), (\ref{adjoint61}) satisfy
$$
\frac{\chi^{\alpha}}{\alpha} \rightarrow 1,\ \frac{u^{\alpha}-1-\alpha z}{\alpha}  \rightarrow u
$$
where the pair $(u, m)$ is the solution of the risk-neutral system  (\ref{bell1}),  (\ref{bell2}),  (\ref{bell21}).

\end{remark}
\section{Stochastic Maximum Principle}
We can derive from the system (\ref{adjoint11})- (\ref{adjoint61}) a stochastic maximum principle. We use the following notation: $X(t), Z(t)$ represent the optimal states and  $V(t)$
represent the optimal control. The probability distribution of $X(t)$ is denoted by $\mathbb{P}_{X(t)}.$ We shall define
 the adjoint processes by $Y(t)=\frac{D_xu(X(t),Z(t),t)}{\chi},\ \eta(t)=\chi(X(t),Z(t),t).$

In fact the real adjoint process is $Y.$ Following the standard notation of stochastic maximum principle, the Hamiltonian is written as
$$
H(X(t),\mathbb{P}_{X(t)},v,Y(t) )=f(X(t),\mathbb{P}_{X(t)},v)+Y(t). g(X(t),\mathbb{P}_{X(t)},v)
$$
and by definition of $F,G,$
$$
F(X(t),\mathbb{P}_{X(t)},v,Y(t) )=f(X(t),\mathbb{P}_{X(t)},V(t) )
$$

$$
G(X(t),\mathbb{P}_{X(t)},v,Y(t) )=g(X(t),\mathbb{P}_{X(t)},V(t) )
$$

In order to state the stochastic maximum principle we compute the It\^o differential of $Y(t).$ We apply Ito's formula to the function $\frac{D_xu(X(t),Z(t),t)}{\chi}.$ 
After tedious calculations, we obtain
$$
\begin{array}{ll}
dY=[-D^2_xu.a.D_x \log \chi+  \frac{D_xu(X(t),Z(t),t)}{\chi} \| \sigma^* D_x \log \chi \|^2\\ -\frac{1}{\chi}\mbox{tr}(D^2_xu. \sigma.D_x\sigma^*)(X(t),Z(t) )]dt\\
-D_xH(X(t),\mathbb{P}_{X(t)},V(t), Y(t))dt\\
-\frac{1}{\chi(X(t), Z(t))}\int_{\mathbb{R}^{n+1}}\chi(\xi,\zeta )D_x\frac{\partial  H}{\partial m}(\xi,m, \frac{D_xu(\xi,\zeta)}{\chi})(X(t))
\\ \quad \ \mu(\xi,\zeta,t) d\xi d\zeta\ dt\\
+
[\frac{D^2_x u}{\chi}-\frac{D_xu}{\chi}. (D_x\log \chi)^*]\sigma(X(t)) dw(t)
\end{array}$$

Set $\eta(t)=\chi(X(t),Z(t),t)=:\chi(t).$ Using It\^o's formula,
$$
d\eta=\eta l dw(t),\  \ \eta(T)=\alpha e^{\alpha[Z(T)+h(X(T),m(T))]},
$$
in which we have set $l(t)=\sigma^* D_x \log \chi(X(t), Z(t), t).$
Let also define $  \Gamma(t)=D^2_xu. \sigma(X(t), Z(t), t).$

$$
\begin{array}{ll}
dX(t)=g(X(t),\mathbb{P}_{X(t)},V(t))dt+\sigma(X(t))dw(t),\\
X(0)=x_0,\\
dZ(t)=f(X(t),\mathbb{P}_{X(t)},V(t))dt,\\
dY=-[\frac{\Gamma(t)l(t)}{\chi(t)}-Y(t) \| l(t)\|^2+\frac{1}{\chi(t)}\mbox{tr}(\Gamma(t)D_x\sigma^*(X(t)))    ]dt\\
-D_xH(X(t),\mathbb{P}_{X(t)},V(t),Y(t)) dt \\ -\frac{1}{\chi(t)} \mathbb{E}[\chi(t)D_x \frac{\partial  H}{\partial m}(X(t),\mathbb{P}_{X(t)},V(t),Y(t) )(X(t))]dt\\
+(-Y(t)l^*(t)+\frac{\Gamma(t)}{\chi(t)})dw(t),\\
Y(T)=D_xh(X(T),\mathbb{P}_{X(T)})+\\ \frac{\alpha}{\chi(T)}\mathbb{E} \left[e^{\alpha[Z(T)+h(X(T),m(T))]} D_x\frac{\partial  h}{\partial m}(X(T),\mathbb{P}_{X(t)})(X(T) \right],\\
d\chi=\chi l dw(t),\  \chi(T)=\alpha e^{\alpha[Z(T)+h(X(T),m(T))]},\\
V(t)\in \arg\min_v \{H(X(t),\mathbb{P}_{X(t)},v,Y(t))  \}
\end{array}$$
The processes $l(t)$ and $\Gamma(t)$ are defined by the fact that $\chi(t)$ and $Y (t)$ are solutions of stochastic backward
differential equations.

\section{ Linear-quadratic risk-sensitive case}
\subsection{Mean-Field Free Case}
Here we assume that $\beta=0$ and
$f(x, v) = x^*Qx + v^*Rv, \ g(x, v) = Ax + Bv,\ $
$h(x)= x^*Q_Tx,\  \sigma(x)=\sigma$

We look for a solution as follows
 $u^{\alpha}(x,t)=  e^{\alpha [x^*\Pi(t)x+\rho(t)]},$ where

\begin{eqnarray}
\frac{d}{dt}\Pi(t)+\Pi(t) A+A^*\Pi(t) \nonumber \\  \label{fail}-\Pi(t)[BR^{-1}B^*-\alpha a]\Pi(t)+Q=0,\\
\Pi(T)=Q_T,\\
\rho(t)=\frac{1}{2}\int_t^T \mbox{tr}[a\Pi(s)]\ ds
\end{eqnarray}

Of course, the Riccati equation (\ref{fail}) may fail to have a solution.

\subsection{Mean-Field Dependent Case}
For linear-quadratic setup we want to solve the following problem (we only consider the one-dimensional case):
\begin{eqnarray}
\label{smp2t}
\left\{
\begin{array}{lll} \inf_{v(\cdot)\in \mathcal{U}}  \mathbb{E} e^{\alpha \left[\frac{1}{2}x^2(T)+\b  E[x(T)]+ z(T)\right]},
\\
\displaystyle{\mbox{ subject to }\ }\\
dx(t)=(ax(t)+bv(t))dt+\sigma dB(t),\\
dz(t)=\frac{1}{2}v^2(t) dt,\\
x(0)=x_{0}, \quad z(0)=0. \\
 \end{array}
\right.
\end{eqnarray}
 With 
$
g(x,m,v):=ax+bv,\,\,\, f(x,m,v):=\frac{1}{2}v^2,\,\,\, $ $ h(x,m):=\frac{1}{2}x^2+\b \int ym(dy), \,\,\, \sigma(x):=\sigma.
$ Note that $g,f,\sigma$ are independent of $m.$ 
The corresponding Hamiltonian
\begin{eqnarray}\label{Ham-1}
\tilde H(x,m,D_xu,D_zu)=\inf_{v}\[ (ax+bv)D_xu+\frac{1}{2}v^2D_zu\]\\
=axD_xu-\frac{1}{2}b^2\frac{(D_xu)^2}{D_zu},
\end{eqnarray}
where the optimal control is 
\be\label{opt-u-LQ}
\bar{v}=-b\frac{D_xu}{D_zu},
\ee
noting that, by Lemma \ref{chi-positive} or (\ref{chi-eq}) below, $D_zu>0.$

The optimal state $(x(t),z(t))$ solves
\be\label{smp2t-opt}
\left\{
\begin{array}{lll} 
dx(t)=(ax(t)-b^2\frac{D_xu}{D_zu})dt+\sigma dB(t),\\
dz(t)=\frac{1}{2}\left(b\frac{D_xu}{D_zu}\right)^2 dt,\\
x(0)=x_{0}, \quad z(0)=0. \\
 \end{array}
\right.
\ee
Its associated infinitesimal generator is
\begin{eqnarray}
\mathcal{A}\psi(x,z)=\frac{\sigma^2}{2}D^2_{x}\psi(x,z)+\left(ax-b^2\frac{D_xu}{D_zu}\right)D_x\psi(x,z)\nonumber \\ \label{generator-opt}+\frac{1}{2}\left(b\frac{D_xu}{D_zu}\right)^2D_z\psi(x,z).
\end{eqnarray}

\subsubsection{The adjoint function $u$} 

The adjoint function $u$ is the solution of
\be\label{u-eq}
\partial_t u(t,x,z)+\mathcal{A}u(t,x,z)=0 ,
\ee
with terminal value 
\be\label{u-eq-T}
u(T, x,z)=e^{\a(z+ h(x,m(T)))}+\a \beta x \int\int e^{\a(\zeta+ h(y,m(T)))} \mu(dy,d\zeta,T).
\ee
In terms of the process  $(x(t),z(t))$ given in  (\ref{smp2t-opt}) we have
\be\label{u-mart}\left\{
\begin{array}{ll}
du(t,x(t),z(t))=\sigma D_xu(t,x(t),z(t))dB(t),\\
u(T, x(T),z(T))=\phi^{\alpha}_T +\a \beta x(T) E[\phi^{\a }_T],\\
\phi^{\alpha}_T:=e^{\a(z(T)+ h(x(T),m(T)))}.
\end{array}
\right.
\ee

\subsubsection{The function $D_zu$} 

Differentiating (\ref{u-eq}) w.r.t. $z$ we obtain the following PDE for $\chi:=D_zu$:
\be\label{chi-eq}
\partial_t \chi+\mathcal{A}\chi=0,\ \ \chi(T)=\a e^{\a(z+ h(x,m(T)))}.
\ee
In terms of the process  $(x(t),z(t))$ given in  (\ref{smp2t-opt}) we have
\be\label{chi-mart}\left\{
\begin{array}{ll}
d\chi(t)=\sigma D_x\chi(t)dB(t),\\
\chi(T)=\alpha \phi^{\a }_T,\ \    \chi(0)=\alpha E[\phi^{\a }_T].
\end{array}
\right.
\ee
\subsubsection{The function $D_xu$}  Differentiating (\ref{u-eq}) w.r.t. $x$ we obtain the following PDE for \\ $\varphi:=D_xu$, (the equality $D_x\chi=D_z\varphi$ is used in the calculation).
\begin{eqnarray}\label{phi-eq}
\partial_t \varphi+\mathcal{A}\varphi=0,\ \\  \varphi(T)=\a x e^{\a(z+ h(x,m(T)))}+\nonumber \\ \a\beta \int\int e^{\a(\zeta+ h(y,m(T)))} \mu(dy,d\zeta,T).
\end{eqnarray}
In terms of the process  $(x(t),z(t))$ given in  (\ref{smp2t-opt}) we have
\be\label{phi-mart}\left\{
\begin{array}{ll}
d\varphi(t)=\sigma D_x\varphi(t)dB(t),\\
\varphi(T)=\alpha x(T) \phi^{\a }_T+\alpha\beta E[\phi^{\a }_T].
\end{array}
\right.
\ee

\subsubsection{Characterization of the optimal control}

Using (\ref{phi-mart}) and (\ref{chi-mart}), by It\^o's formula, the process $p:=\frac{D_xu}{D_zu}=\frac{\varphi}{\chi}$ satisfies
\be\label{pq-opt}\left\{
\begin{array}{ll}
dp(t)=-\{a p(t)+\ell(t) q(t)\}dt+q(t)dB(t),\\
p(T)=\frac{\varphi(T)}{\chi(T)}=x(T)+\beta \frac{\chi(0)}{\chi(T)},
\end{array}
\right.
\ee
where
\be\label{ell}
\ell(t):=\sigma\frac{D_x\chi}{\chi}=\sigma\frac{D_{xz}u}{D_zu}
\ee
and
\be\label{q-phi}
q(t)=\sigma\frac{D_x\varphi(t)}{\chi(t)}-p(t)\ell(t).
\ee

The process $(p,q,l)$ has an explicit solution in terms of $\chi(0)$ and deterministic function $\pi,\omega$ and is given by
$
p(t)= \pi(t) x(t) +\beta \omega(t) \frac{\chi(0)}{\chi(t)}, \ $ $ q(t)= \sigma \pi(t) -\beta l(t) \omega(t) \chi(0) \chi^{-1}(t).
$
where $\pi,\omega$ solve Riccati equations.

$$\left\{
\begin{array}{ll}
\dot{\pi}+2a\pi -b^2\pi^2+\a \beta \sigma^2 \gamma =0,\ \pi(T)=1\\
\dot{\omega}+(a-b^2\pi)\omega=0,\ \omega(T)=1.
\end{array}
\right.
$$

The expected value $y(t)=E[x(t)]$ of the optimal state  solves the ODE
$$
\dot{y}= a y-b^2 E[p]=(a-b^2\pi)y- \beta  b^2 \omega  \chi(0)E[\chi^{-1}(t)].
$$

The optimal value of the problem is $v^{\alpha}:=\frac{D_zu (0,x_0,0)}{\alpha}=\frac{\chi(0)}{\alpha}.$
\subsection{Approximation of the risk-sensitive value}

In this section we assume that $\beta\neq 0$ is small. In the previous section we provided an explicit solution of the value function in terms of $\frac{\chi(0)}{\alpha}=\mathbb{E}\phi^{\alpha}.$
However, $\chi(0)$ needs to be calculated.  When $\beta=0$ this was explicitly computed from $\pi(0),\omega(0)$ (which were denoted $\Pi(0),\rho(0)$).
The optimal value of the adjoint function is $u_0(x,z,t)=e^{\alpha(z+\frac{1}{2}\Pi(t)x^2+\rho(t))}.$

Now, when $\beta\neq 0$ we aim to approximate $u_{\beta}(T):=\mathbb{E}\phi^{\alpha}_T=e^{\alpha \beta y(T)} \mathbb{E}\[e^{\alpha [z(T)+\frac{1}{2}x^2(T)]}\].$
For $\beta$ small enough we test the ansatz $u_{\beta}(x,z,t)=e^{\alpha \beta y(T)} \tilde{u}(x,z,t)$
where $y(T)=y_{\beta}(T)\sim y_0(T)=\bar{x}_0 e^{\int_0^T [a-b^2\Pi(t)]dt },$
$$
\mathbb{E}\[ e^{\alpha [z(T)+\frac{1}{2}x^2(T)]}\]\sim \int_{x} \int_{z}u_0(x,z,0)\mu(x,z,0)dxdz,
$$
which is expressed as
$\int_{x} u_0(x,0,0)m_0(x)dx=\int_x e^{\frac{1}{2}\Pi(0) x^2+\rho(0)}m_0(x)dx.$

The function $\tilde{u}$ solves the partial differential equation
\begin{equation} \label{tilde1}
\frac{\partial \tilde{u}}{\partial t}+\frac{a}{2}D_{xx}\tilde{u}+ax D_x\tilde{u}-\frac{b^2}{2}\frac{|D_x\tilde{u}|^2}{D_z\tilde{u}}=0,
\end{equation}
and $\tilde{u}(x,z,T)=e^{\alpha(z+\frac{1}{2}x^2)}+\alpha \beta \lambda(T),$ where $\lambda(T)=\mathbb{E}_{\mu(.,T)} e^{\alpha(z(T)+\frac{1}{2}x^2(T)}.$
We will check that 
$\tilde{u}(x,z,t)\sim u_0(x,z, t) + \alpha \beta x \omega(t)\lambda(T)=\hat{u}(x,z,t),$
for a convenient function of time  $ \omega(t)$ to be determined. By choosing $ \omega(t)$ as 
$$
 \omega(t)= e^{\int_t^T [a-b^2\Pi(s)] ds},
$$
we  easily check that
$$
\frac{\partial \hat{u}}{\partial t}+\frac{a}{2}D_{xx}\hat{u}+ax D\hat{u}-\frac{b^2}{2}\frac{| D_x \hat{u}|^2}{D_z\hat{u}}=\frac{b^2}{2}\frac{\alpha\beta^2}{\omega^2}{u_0}.
$$
This means that the function $\hat{u}$ is consistent with  (\ref{tilde1}) around $t=T.$ When $\beta^2$ can be neglected, the approximation of $\tilde{u}$ by $\hat{u}$ 
becomes exact. If we accept this approximation, we still need to fix the value of $y(T)$ and $\lambda(T).$ This involves calculations using the probability measure $\mu$ which 
solve the Fokker-Planck-Kolmogorov equation
\begin{equation} \label{tilde2}
\frac{\partial }{\partial t} \mu-\frac{a}{2}D_{xx}\mu+ \frac{\partial}{\partial x}( (ax-b\frac{D_x u}{D_z u})\mu )+\frac{b^2}{2}\partial_z\left[\mu \frac{|D_x u|^2}{|D_z u|^2}\right]=0.
\end{equation}
From the approximation above it follows that the ratios  $\frac{D_xu}{D_zu}$ and $\frac{D_x\hat{u}}{D_z\hat{u}}$ are close to each other: $$\frac{D_xu}{D_zu}=\frac{D_x\tilde{u}}{D_z\tilde{u}}\sim  \frac{D_x\hat{u}}{D_z\hat{u}}= \Pi(t) x+\beta \frac{\lambda(T) 
\omega(t)}{u_0(x,z,t)}.
$$

Neglecting the term in $\beta$ of the last term
$\frac{D_xu}{D_z u}\sim \Pi(t) x$
and hence $\mu$ will be approximated by $\hat{u}$ as
\begin{equation} \label{tilde3}
\frac{\partial }{\partial t}\hat{\mu}-\frac{a}{2}D_{xx}\hat{\mu}+ \frac{\partial}{\partial x}( x(a-b\Pi)\hat{\mu} )+\frac{b^2\Pi^2x^2}{2}D_z\hat{\mu}=0,
\end{equation}
and $\mu(x,z,0)=\hat{\mu}(x,z,0)=m_0(x)\otimes \delta_0(z).$
We immediately deduce that the expected value of $x(t)$ as $y(t)=\int_{x\in \mathbb{R}}\int_{z\in \mathbb{R}}x\mu(x,z,t)dxdz$ is solution of 
the ordinary differential equation
$$
\dot{y}=y (a-b^2\Pi),\ \  y(0)=\bar{x}_0=\int x m_0(x)dx.
$$

Next, writing 
$$
- \frac{\partial u_0}{\partial t}-\frac{a}{2} D_{xx}u_0=x(a-b^2\Pi)D_x u_0+\frac{b^2\Pi^2 x^2}{2} D_z u_0
$$ and testing with $\hat{\mu}$ we obtain 
\begin{eqnarray}
\lambda(T)&=&  \mathbb{E}\[ e^{\alpha [z(T)+\frac{1}{2}x^2(T)]}\]\\
&\sim & \int_{x,z}u_0(x,z,0)\mu(x,z,0)dxdz\\
&=& \int_{x} u_0(x,0,0)m(x,0)dx\\ 
&\sim & \int_{x\in \mathbb{R}} e^{\alpha(\frac{1}{2}\Pi^2(0)x^2+\rho(0))} m_0(x)dx.
\end{eqnarray}
This completes the approximation in $\beta.$

\section{Conclusion}
In this paper we have developed a risk-sensitive mean-field-type optimal control framework. We have considered performance functionals and coefficients that are allowed to depend on some functional of the law as well as the state and control processes.  We derived optimality equations in infinite dimensions connecting dual functions  associated with Bellman functional to the adjoint processes of the  Pontryagin stochastic maximum principle. 
\section*{Acknowledgement}
The first author is supported by grants from the National Science Foundation (1303775 and 1612880), the Research Grants Council of the Hong  Kong Special Administrative Region (City U 500113 and 11303316).
 The research of the second author is supported by grants from the  Swedish Research Council. The research of the third author is supported by U.S. Air Force Office of Scientific Research.

\end{document}